\documentclass[11pt,a4paper]{amsart}
\usepackage{graphicx}
\usepackage{amssymb}
\usepackage[all]{xy}
\usepackage{amsmath}
\usepackage[french,english]{babel}
\usepackage{amssymb,amsmath,amsfonts,amsthm,amscd}
\usepackage{indentfirst,graphicx}
\usepackage[font={scriptsize},captionskip=5pt]{subfig}
\def\pd#1#2{\frac{\partial#1}{\partial#2}}

\newtheorem{theorem}{Theorem}[section]
\newtheorem{exa}[theorem]{Example}

\newtheorem{rem}[theorem]{Remark}
\newenvironment{remark}{\begin{rem} \em}{\end{rem}}
\newtheorem{cor}[theorem]{Corollary}

\newtheorem{tont}[theorem]{Definition}

\newtheorem{lem}[theorem]{Lemma}
\newenvironment{Lemme}{\begin{lem} \em}{\end{lem}}

\newtheorem{Prop}[theorem]{Proposition}
 
\let\:=\colon

\newcommand{\cO}{{\mathcal O}}

\def\sdep #1{#1^\dagger}

\begin{document}

\title [$A_f $ and relative conormal spaces]{The $A_f$  condition and relative conormal spaces for functions with non-vanishing derivative }

\author[T. Gaffney]{Terence Gaffney}\thanks{T.~Gaffney was partially supported by PVE-CNPq Proc. 401565/2014-9}
 \address{T. Gaffney, Department of Mathematics\\
  Northeastern University\\
  Boston, MA 02215}

\author[A. Rangachev]{Antoni Rangachev}\thanks{A.~Rangachev was partially supported by University of Chicago FACTTS grant ``Conormal and Arc Spaces in the Deformation Theory of Singularities''}
\address {A. Rangachev, Department of Mathematics\\
  University of Chicago\\
  Chicago, IL 60637\\
Institute of Mathematics and Informatics\\
Bulgarian Academy of Sciences\\
Akad. G. Bonchev, Sofia 1113, Bulgaria}

\begin{abstract}
We introduce a join construction as a way of completing the description of the relative conormal space of an analytic function on a complex analytic space that has a non-vanishing derivative at the origin. Then we show how to obtain a numerical criterion for Thom's $A_f$ condition.
\end{abstract}

\maketitle

\selectlanguage{english}

\section{Introduction}\label{intro} Throughout the paper when dealing with complex analytic germs we will work with suitably small representatives of these germs. Let $(X,0)$ be a reduced complex analytic germ in $(\mathbb{C}^m,0)$. Define the {\it conormal space} $C(X)$ of $X$ as the closure in $X \times \mathbb{P}^{m-1}$ of the set of pairs $(x,H)$ such that $x$ is a smooth point of $X$ and $H$ is a tangent hyperplane to $x$. Suppose $f$ is a nonconstant analytic function on $X$. Then the {\it relative conormal space} $C(X,f)$ is defined as the closure in $X \times \mathbb{P}^{m-1}$ of the set of pairs $(x,H)$ such that $x$ is a smooth point of the level set $f^{-1}(f(x))$ and $H$ is a tangent hyperplane to $f^{-1}(f(x))$ at $x$. 

In his Thm.\ 4.2 in \cite{Massey2}, Massey describes the fiber of the relative conormal space $C(X,f)$  over $0$ assuming $f \in m_{X,0}^2$, where $m_{X,0}$ is the maximal ideal of $\mathcal{O}_{X,0}$. Our Thm.\ \ref{conormal fiber} gives description of this fiber in the case $f \not \in m_{X,0}^2$. More precisely, let $\tilde{f}$ be an extension of $f$ to $\mathbb{C}^m$. In addition to the irreducible components of the fiber $C(X,f)$ over $0$ described by Massey we show that whenever $f \not \in m_{X,0}^2$ there are components which are the join in $\mathbb{P}^{m-1}$ of the point $d\tilde{f}(0)$ and the irreducible components of the fiber of $C(X)$ over $0$. 

We apply Thm.\ \ref{conormal fiber} to the numerical characterization of Thom's $A_f$ condition (see \cite{Thom}), which is a relative stratification condition for the study of functions and mappings on stratified sets. It plays an important role in Thom's second isotopy theorem, and provides a transversality condition in the development of the Milnor fibration. 

Assume $(Y,0)$ is a smooth subgerm of $(X,0)$ such that $f(Y)=0$. We say that the $A_f$ condition holds for the pair $(X-Y,Y)$ at $0$ if the fiber of $C(X,f)$ over $0$ lies in conormal space $C(Y)$ of $Y$. We say that $A_f$ holds along $Y$ if it holds at every point of $Y$, or if $C(X,f)|Y\subset C(Y)$. The $A_f$ condition is known to hold generically along $Y$ by a result of Hironaka \cite{Hironaka}. So it is important to understand the fiber of $C(X,f)$ over the origin and its relation to $C(X,f)|(Y-0)$.

Set $n:=m-k$. Choose an embedding of $(X,0)$ in $\mathbb{C}^{m}=\mathbb{C}^{n} \times \mathbb{C}^{k}$, so that $(Y,0)$ is represented by $0 \times V$, where $V$ is an open neighborhood of $0$ in $\mathbb{C}^k$. Let $\mathrm{pr}: \mathbb{C}^{n} \times \mathbb{C}^{k} \rightarrow \mathbb{C}^{k}$ be the projection. View $X$ as the total space of the family $\mathrm{pr}_{|X}: (X,0) \rightarrow (Y,0)$. For each closed point $y \in Y$ set $X_y: = X \cap \mathrm{pr}^{-1}(y)$.

Assume that $X$ and the fibers $X_y$ are equidimensional and $Y$ is the singular locus of $X$. Further, assume that $f(Y)=0$ and $f \notin m_{Y}^2$, where $m_{Y}^2$ is the ideal of $Y$ in $\mathcal{O}_{X,0}$. Then Thm.\ \ref{conormal fiber} and the lemma preceding it tell us that the irreducible components of $C(X,f)|Y$ are of two types: ``big components'' and ``small components'' which are the join of $d\tilde{f}(0)$ with the irreducible components of the fiber of $C(X)$ over the origin. By numerical control on the fibers $X_y,f_y$ we can ensure that the fiber of $C(X,f)$ over $0$ is of minimal dimension. Then it is easy to see that all ``big components'' are contained in $C(Y)$. The ``small components'' are hard to control, because their dimension may be small.

Denote by $C(X_0)'$ the closure of the set of pairs $(x,H)$ where $x$ is a smooth point in $X_0$ and $H$ is a tangent hyperplane of $X$ at $x$. To understand the nature of the ``small components'' one needs to understand the relation between 
$C(X)|X_0$ and $C(X_0)'$. In Thm.\ \ref{components} we give a dimensional condition which ensures that $C(X)|X_0$ and $C(X_0)'$ are the same up to embedded components. Finally, assuming that $X$ satisfies the infinitesimal Whitney A fiber condition along $Y$, which is a much weaker version of Whitney condition A, we show that the components of the fiber $C(X_0)'$ over the origin are contained in $C(Y)$. All this is the content of Thm.\ \ref{main}, which is the main result of the paper, and Thm.\ \ref{determinantal}.

\section{Relative Conormal Spaces}
We begin with  some constructions and notation. Let $(X,0)$ be a reduced complex analytic germ in $(\mathbb{C}^m,0)$ and let $U$ be an open set of  $\mathbb{C}^m$ containing a representative of $(X,0)$. Denote by  $T_{X}^{*}U$ the space obtained by taking the closure of the conormal vectors to the smooth part of $X$ in $\mathbb{C}^m\times {\mathbb{C}^m}^*$. As the fibers of $T_{X}^{*}U$ over points in $X$ are invariant under multiplication by elements from $\mathbb{C}^{*}$, we may projectivize $T_{X}^{*}U$ with respect to vertical homotheties of $T_{X}^{*}U$ and work with $\mathbb{P}(T_{X}^{*}U)$. This is precisely the conormal space $C(X)$ described in the introduction. 

Suppose $f$ is a function on $X$ and $\tilde{f}$ is an extension of $f$ to $\mathbb{C}^m$. The relative conormal space $C(X,f)$ of $X$ with respect to $f$ as defined in the introduction can be obtained as follows. Let $T_{f}^{*}U$ be the closure of all $(x,\eta)$ in $T_{X}^{*}U$ where $x$ is a smooth point in $X$ and $\eta(T_{x}X \cap \mathrm{ker}(  d\tilde{f})) = 0$. Then $C(X,f)$ is the projectivization of $T_{f}^{*}U$. Note that $C(X,f)$ does not depend on the choice of extension $\tilde{f}$ of $f$ (cf.\ Sct.\ 5 in \cite{G-K}).  Denote by $c \: C(X,f) \rightarrow X$ the structure morphism. For a point $x \in X$ denote by $C(X)_{x}$ and $C(X,f)_x$ the fibers of $C(X)$ and $C(X,f)$ over $x$ respectively. 

The  differential $d\tilde{f}$ of
$\tilde{f}$ defines an embedding of $X$ in  $\mathbb{C}^m\times {\mathbb{C}^m}^*$ by the graph map. Let $z_1, \ldots, z_m$ be coordinates on $U$ and $w_1, \ldots, w_n$ be the cotangent coordinates. Then the blowup of $T_{X}^{*}U$ along the image of the graph map is the blowup of $T_{X}^{*}U$ by the ideal $(w_1- \frac{\partial \tilde{f}}{\partial z_1}, \ldots, w_n - \frac{\partial \tilde{f}}{\partial z_m})$ in $T_{X}^{*}U$. We denote this blowup by $\mathrm{Bl}_{\mathrm{im} \ d\tilde{f}}T_{X}^{*}U$. Thus, the blowup is contained in $X \times {\mathbb{C}^m}^* \times \mathbb{P}^{m-1}$. Denote the exceptional divisor of this blowup by $E_{\tilde{f}}$. The projection of this exceptional divisor to $X$ is the {\it singular locus} of  $f$ on $X$ denoted by $S(f)$.   

Let $\pi : X \times {\mathbb{C}^m}^* \times \mathbb{P}^{m-1} \rightarrow X \times \mathbb{P}^{m-1}$ denote the projection. Then $\pi (E_{\tilde{f}})$ is independent of the extension of $\tilde{f}$ of $f$ by Cor.\ 2.12 in \cite{Massey2}. The following result describes the relation between  $C(X,f)$ and $\mathrm{Bl}_{\mathrm{im} \ d\tilde{f}}T_{X}^{*}U$. 
 
\begin{Lemme}\label{exceptional divisor} The following holds. 
\begin{itemize}
    \item [(\rm{i})] $E_{\tilde{f}} \cong \pi (E_{\tilde{f}}).$
    \item [(\rm{ii})] Suppose $S(f) \subset X_{\mathrm{sing}}$. Then $\pi (E_{\tilde{f}}) \subset c^{-1}(X_{\mathrm{sing}})$.  
\end{itemize}
\end{Lemme}
\begin{proof}
Part $\rm{(i)}$ is due to Massey (see the paragraph preceding Lemma 2.6 in \cite{Massey2}). If $(x,w,\eta)$ is a point in $E_{\tilde{f}}$, then $w=d\tilde{f}(x)$, so $\pi$ induces and isomorphism between $E_{\tilde{f}}$ and $\pi(E_{\tilde{f}})$. 

Consider $\rm{(ii)}$. By Lemma 2.6 in \cite{Massey2} (see also the proof of Thm.\ \ref{conormal fiber} below) it follows that $\pi (E_{\tilde{f}}) \subset C(X,f)$. But $E_{\tilde{f}}$ is supported over points $(x,w_{x})$ for which $w_x = d\tilde{f}(x)$ where $w_x$ is a conormal at $x$. But by hypothesis $w_x = d\tilde{f}(x)$ can happen only over singular points of $X$, which proves the claim. 
\end{proof}
We will also use the join operation, which we now describe. Given a point $a$ of projective space 
$\mathbb{P}^{m-1}$ and a subset $V$ of 
$\mathbb{P}^{m-1}$ 
distinct from the point, the {\it join} of $a$ and $V$ consists of 
the set of points
on all lines joining $a$ and the points of $V$. If $a$ is a point of $V$ and $V \neq a$, then the operation is still
well defined; one merely takes the join of $a$ and $V-a$,  and then takes the closure of this set.
It is easy to see that if $V$ is an analytic set and $a$ lies in $V$, then the join contains
the tangent cone to $V$ at $a$.  If $V$ is analytic, then so is the join, for we can view the join as the 
inverse image of the projection of $V-a$ to $\mathbb{P}^{m-2}$ from the point $a$. 
We denote the join of $a$ and $V$ by $a*V$. 

Let $C,x$ be a curve on $X,x$. Let $D_i$, $i=1,2$  be two lifts of $C$ to $X \times {\mathbb P}^{m-1}$, Denote by $D_{i,p}$ the fiber of $D_i$ over $p\in C$. Suppose $D_{1,p}\ne D_{2,p}$ for $p$ near $x$. Denote by $(D_1*D_2)_C$ the family of lines parameterized by $C$ whose fiber over $p$ is $D_{1,p}*D_{2,p}$.

Let $x$ be a point in $X$. Suppose $f\notin {m^2}_{X,x}$. Denote by $\langle d\tilde{f}(x) \rangle$ the  point of $\mathbb{P}^{m-1}$ determined by $d\tilde{f}(x)$. Denote the join of $\langle d\tilde{f}(x) \rangle$ and a subset $V$
of $\mathbb{P}^{m-1}$  by $d\tilde{f}(x)*V$ as well. It is easy to check that $d\tilde{f}(x)*C(X)_x$ is 
independent of the choice of extension of $f$ to the ambient space. If $f\in {m^2}_{X,x}$, then by convention
$d\tilde{f}(x)*C(X)_x$ is empty.

In the next theorem we will be working with limits along curves, so we discuss this a little. Given $G\in \mathcal{O}_{\mathbb{C}}^p$, $G(t)\ne 0$, for $t\ne 0$, the {\it limit direction of $G$ at $t=0$} is $\mathop {\lim }\limits_{t\to 0} \langle G(t) \rangle$, which is a point of $\mathbb{P}^{m-1}$. 

We can find the projective limit of $G$ by working directly with $G$ as follows. For any $g \in \mathcal{O}_{\mathbb{C}}$ denote by $o(g(t))$ the order of $t$ in $g(t)$. If $G(t) \in  \mathcal{O}_{\mathbb{C}}^p$, then $o(G(t))$ is the minimum of the orders of the component functions $g_i$ of $G(t)$. If $o(G(t))=k$, then the $p$-tuple whose entries are the coefficients of the degree $k$ terms of the  $g_i$ is the {\it leading term of $G$}. If we denote the leading term of $G$ by $L(G)$ then $\langle L(G) \rangle$ is the limit direction of $G$ at $t=0$. We can compute $L(G)$ as 
$$\mathop {\lim }\limits_{t\to 0}  {1\over{t^k}}(G(t)),$$
where $k$ is the order of $G$.

The next proposition is the key to the description here of the relative conormal space. It grew 
out of an attempt to improve on some work of Massey (cf.\ Thm.\ 3.11 in \cite{Massey1}). In particular, 
the idea of using the blow-up
of the graph of the differential of $\tilde{f}$ to study the relative conormal space is an idea we learned from him.

\begin{theorem} \label {conormal fiber} Suppose $(X,0)$ is the germ of a reduced complex analytic set and $f\: X\to \mathbb{C}$ is a submersion
on a Zariski open and dense subset of $X$. Then for each point $x \in X$ we have the set-theoretic equality  $$C(X,f)_x=(\pi(E_f))_x\cup d\tilde{f}(x)*C(X)_x$$\end{theorem}
\begin{proof} The case $f \in m_{X,x}^2$ is part of Thm.\ 4.2 in \cite{Massey2}. 

We first show that $C(X,f)_x$ contains $ d\tilde{f}(x)*C(X)_x$. Suppose $d\tilde{f}(x) \neq 0$.
Suppose $H\in C(X)_x$ and  $H\ne \langle d\tilde{f}(x) \rangle$.  
There exists a curve
$$\noindent \phi=(\phi_1,\phi_2)\:(\mathbb{C},0)\to (X\times \mathbb{P}^{m-1},x\times H),$$ such that
the hyperplane $\phi_2(t)$ is tangent
to $X$ at $\phi_1(t)$, and $\phi_1(t)\in X-X_{\mathrm{sing}}-S(f)$ for  $t\ne0$, where $S(f)$ is the critical locus of $f$.  Denote the image of $\phi_1$ by $C$.
For $t$ sufficiently small, $t\ne 0$, 
we can assume that the augmented Jacobian module, which is the module generated over $\mathcal{O}_{X,x}$ by the columns of the Jacobian matrix of $(G,\tilde{f})$, has maximal rank because $\phi_1(t)\in X-X_{\mathrm{sing}}-S(f)$ for $t\ne0$. Then $\langle d\tilde{f}(\phi_1(t)) \rangle$ and $\phi_2(t)$ give two lifts   of $C$ to ${\mathbb P}^{m-1}$ and since the rank of the augmented Jacobian module is maximal along $C$, then  $(\langle d\tilde{f} \rangle*\phi_2)_C$ is well-defined. Since both $C(X,f)$ and  $(\langle df \rangle*\phi_2)_C$ are Zariski closed and a Zariski open subset of the second lies in the first, then the second lies in the first as well. This implies that $d\tilde{f}(x)*H$ is in $C(X,f)_{x}$.

The rest of the proof is related to the fibers of $C(X,f)$ or $\mathrm{Bl}_{\mathrm{im} \ d\tilde{f}}T_{X}^{*}U$,  so it is convenient to work along curves and take limits. Giving a curve on $C(X,f)$ at smooth points of $f$ on $X$ amounts to giving smoothly varying linear combinations of the rows of the Jacobian matrix of $F=(G,\tilde{f})$ where $G: (\mathbb{C}^{m},0) \rightarrow (\mathbb{C}^{p},0)$ and $X = G^{-1}(0)$ and projectivising. We can make this precise as follows: if  $(x,H)\in C(X,f)_x$ then there exist curves $\phi$, $\psi$ such that $\phi\:(\mathbb{C},0)\to (X,x)$, and 
$$\psi=(\psi_1,\psi_2)\:(\mathbb{C},0)\to \mathbb{C}^p\times\mathbb{C}.$$ Then $H$ is the projective limit of the curve 
\begin{equation}\label{limit}
 {\psi_1 (t)\cdot  DG(\phi (t))-\psi_2(t) d\tilde{f}(\phi (t))}
\end{equation}
where $\psi_2(t)$ is taken with a minus sign for convenience of comparison with the blow-up construction. 

We can do something similar for $l\in \pi(E_{\tilde{f}})_x$. Namely, we can use $\phi$ and $\psi$ as before with $\psi_2=1$. Then $(x,l)$ is a point of $\pi(E_{\tilde{f}})_x$ if and only we can find curves 
$\phi$, $\psi$ such that $\phi(0)=x$ and  $l$ is the projective limit of the curve 


$$\psi_1 (t)\cdot  DG(\phi (t))-d\tilde{f}(\phi (t)).$$

This description shows that  $\pi(E_{\tilde{f}})_x$ is also contained in $C(X,f)_x$. 

Now suppose $(x,H)\in C(X,f)_x$. Then there exist curves $\phi$, $\psi$ such that $H$ is the projective limit of a curve of type (\ref{limit}). 

We deal separately with the cases where $f\in {m^2}_{X,x}$ and $f\notin {m^2}_{X,x}$.

Assume $f\in {m^2}_{X,x}$ and that $\tilde{f}$ is chosen in such a way so that  $d\tilde{f}(x)=0$.

If
$$o(\psi_2(t))<o(\psi_1(t)\cdot DG(\phi (t))),$$ then 
${{\psi_1(t)}\over{\psi_2}(t)}
\cdot  DG(\phi (t))$
gives a lift of $\phi$ to $T_X^*U$ and the projective limit of 
${{\psi_1(t)}\over{\psi_2(t)}}
\cdot  DG(\phi (t))- d\tilde{f}(\phi (t))$ is $H$, showing that $H$ lies in $\pi(E_{\tilde{f}})_x$. 

We have $o(d(\tilde{f}(\phi(t)) \geq 1$, because $f \in m_{X,x}^2$. Suppose 
$$o(\psi_2(t))\ge o(\psi_1(t)\cdot DG(\phi (t))).$$ 
Then
$$o(\psi_2(t)d\tilde{f}(\phi(t)))>o(\psi_1(t)\cdot DG(\phi (t))),$$
so $H$ is the projective limit of $ \psi_1(t)\cdot DG(\phi (t)) $. Hence $H$ is a limiting tangent hyperplane to $X$.  

There are two subcases, depending on the order $d\tilde{f}(\phi)$.

If the order of the components of $d\tilde{f}(\phi)$ is greater than $1$, 
then
we use the same $\phi$ but replace $\psi_1(t)$ by $\psi_1(t)/t^k$ where $k$ is chosen so that
the order of $(\psi_1(t)/t^k)\cdot DG(\phi(t))$ is greater than $0$, but less 
than the order of $d\tilde{f}(\phi (t))$. We may take $k=$max$\{o(\psi_1(t)\cdot DG(\phi (t))-o(d\tilde{f}(\phi (t)))-1, 0\}$. Then again 
$(\psi_1(t)/t^k)\cdot DG(\phi(t))$ provides a lift of $\phi$ to $T_X^*U$, and
the projective limit of $  (\psi_1(t)/t^k)\cdot DG(\phi(t))-d\tilde{f}(\phi(t))$ is again $H$.  If the order of 
the components of $df(\phi(t))$ is  $1$, then we re-parameterize $\phi(t)$ so that the order of
 $d\tilde{f}(\phi(t))$ is again greater than $1$ and repeat the argument. This finishes the first case.

Now suppose $f\notin {m^2}_{X,x}$. If 
$o(\psi_2(t)) < o(\psi_1(t)\cdot DG(\phi(t)))$, then $H=\langle d\tilde{f}(x) \rangle$. If the order
relation is reversed, then $H$ is in $C(X)_x$.  In either case $H\in d\tilde{f}(x)*C(X)_x$,
unless $C(X)_x= \langle d\tilde{f}(x) \rangle$.

If $o(\psi_2(t))=o(\psi_1(t)\cdot DG(\phi (t)))$, then $H\in d\tilde{f}(x)*C(X)_x$ unless
$$\mathop {\lim }\limits_{t\to 0}{{\psi_1(t)}\over{\psi_2(t)}}
\cdot  DG(\phi (t))=d\tilde{f}(x).$$  In this case we can again get a lift of  $\phi$ to $T_X^*U$,
using ${{\psi_1(t)}\over{\psi_2(t)}}
\cdot  DG(\phi (t))$, so again $H$ is in $\pi(E_{\tilde{f}})_x$.

It remains to deal with the case where $C(X)_x= \langle d\tilde{f}(x) \rangle$. We need to show that 
$ \langle d\tilde{f}(x) \rangle$ lies in $\pi(E_{\tilde{f}})_x$.  
Since the dimension of $C(X)_x$ is zero, $X$ must be a hypersurface, and by a change of coordinates we may assume $f$ is a linear form.  There exists $\phi$ such that the projective limits of $DG(\phi(t))$ and of $tDG(\phi(t))$ are both $\langle
d\tilde{f}(x)\rangle$.  Let $k=o(DG(\phi(t)))$. Then

$$\mathop {\lim }\limits_{t\to 0} \langle {{(1+t)}\over {t^k}}DG(\phi(t))-
d\tilde{f}(\phi(t)) \rangle= \langle d\tilde{f}(x) \rangle$$ which completes the proof.\end{proof}

In checking the $A_f$ condition at the origin in a family, we need to show that $C(X,f)_{0}$ consists of hyperplanes which contain the tangent plane to $Y$ at the origin. The previous theorem shows that the components of $C(X,f)_{0}$ are of two types--blowup components and join components. Blowup  components have large dimension, and can be detected and controlled numerically. However,  $C(X)_0$ may contribute small components of join type when $f \not \in m_{X,0}^2$. In the next section we prove a theorem which shows that these join components can be controlled using the fiber $X_0$. 

\section{Fibers of generalized conormal spaces}
Let $h:(X,0) \rightarrow (Y,0)$ be a complex analytic family such that $X$ is equidimensional and for each closed point $y \in Y$ the fibers $X_y$ are equidimensional of positive dimension $d$. Suppose $Y$ is irreducible and Cohen--Macaulay. 

The purpose of this section is to understand the relation between the closed subscheme of the conormal $C(X)$ which set-theoretically consists of limits of hyperplanes through points of $X_0$, and the fiber $C(X)_0$ over $0 \in Y$ of the conormal $C(X)$. 
Our treatment is more general. The conormal $C(X)$ is the $\mathrm{Projan}$ of the Rees algebra of the Jacobian module of $X$ (cf.\ Sct.\ 1.5 in \cite{KT}). Instead of working with conormal spaces, we work below with $\mathrm{Projan}$s of Rees algebras of modules.

Let $\mathcal{F}:=\mathcal{O}_{X}^p$ be a free module of rank $p \geq 1$. Let $\mathcal{M}$ be a coherent submodule which is free of rank $e$ off a closed subset $S$ of $X$. Further, assume $S$ is finite over $Y$. Set $r:=d+e-1$. 

Form the symmetric algebra $\mathrm{Sym}(\mathcal{F})$ of $\mathcal{F}$ and the Rees algebra $\mathcal{R}(\mathcal{M})$ of $\mathcal{M}$ which is the subalgebra of $\mathrm{Sym}(\mathcal{F})$ generated by $\mathcal{M}$ placed in degree $1$. Denote the $k$th graded components of these algebras by $\mathcal{F}^{k}$ and $\mathcal{M}^{k}$ respectively. Given a closed point $y \in Y$ denote by $\mathcal{M}^{k}(y)$ the image of $\mathcal{M}^{k}$ in the free $\mathcal{O}_{X_{y}}$-module $\mathcal{F}^{k}(y)$.

Set $C:=\mathrm{Projan}(\mathcal{R}(\mathcal{M}))$. Denote by $c \colon C \rightarrow X$ be the structure morphism. Let $y$ be a closed point in $Y$. Set $C(y):=\mathrm{Projan}(\mathcal{R}(\mathcal{M}(y)))$ and denote by $C_y$ the fiber of $h \circ c$ over $y \in Y$. For an irreducible component $V$ of $C_y$ we say it is {\it horizontal} if it surjects onto an irreducible component of $X_y$ or we say it is {\it vertical} otherwise. 

\begin{theorem}\label{components}
Suppose $\dim c^{-1}0 <r$. Then there exists a Zariski open neighborhood $U$ of \ $0$ in $Y$ such that for each 
$y \in U$ the irreducible components of $C_y$ are horizontal. Furthermore, if $\mathcal{M}$ is a direct summand of $\mathcal{F}$ locally off $S$, then we have an equality of fundamental cycles
\begin{equation}\label{fund. cycl.}
[C_y]=[C(y)]. 
\end{equation}
\end{theorem}
\begin{proof}
Because $h$ has equidimensional fibers of dimension $d$ and $X$ is equidimensional, then $\dim X = d+\dim Y$. 
Also, by assumption $d>0$ and $S$ is finite over $Y$. Thus $S$ is nowhere dense in $X$. Let $x$ be a point in $X$ with $x \not \in S$. Because the formation of Rees algebra commutes with flat base change we have $\mathcal{R}(\mathcal{M})_x = \mathcal{R}(\mathcal{M}_x)$. But $\mathcal{M}_x$ is free of rank $e$ becase $x \not \in S$. Thus $\mathcal{R}(\mathcal{M}_x) = \mathrm{Sym}(\mathcal{O}_{X,x}^{e})$ whence $\dim c^{-1}x =e-1$. The dimension formula applied for each irreducible component of $C$  yields that $C$ is equidimensional and $\dim C = r+ \dim Y$. 

Let  $C_{y}'$ be an irreducible component of $C_y$. Because $Y$ is Cohen--Macaulay locally at each closed point $y \in Y$, then by Krull's height theorem $C_{y}'$ is of codimension at most $\dim Y$. Because $C$ is of finite type over the complex numbers we get
\begin{equation}\label{dim. ineq.}
\dim C_{y}' \geq r.
\end{equation}

Replace $X$ with Zariski neighborhood of $0$ so that $\dim c^{-1} x <r$ for each point $x \in X$. Let $U$ be a Zariski open subset of the image of $h$ that contains $0 \in Y$. Let $y$ be a point in $U$. Suppose that $c$ maps $C_{y}'$ to a point $\zeta \in X_y$. Then  $C_{y}' \subset c^{-1}\zeta$. But  $\dim c^{-1} \zeta <r$ which contradicts with (\ref{dim. ineq.}). But $S_y$ is zero-dimensional. Thus there exists a Zariski open dense subset $Z_y$ of $C_{y}'$ whose image under $c$ misses $S_y$. 

Let $\zeta \in c(Z_y)$. Then $\mathcal{M}_{\zeta}$ is free of rank $e$. Thus $\dim c^{-1}\zeta = e-1$. By the dimension formula 
$$\dim C_{y}'=\dim c(C_{y}')+e-1.$$ But $\dim C_{y}' \geq r$. So $\dim c(C_{y}') \geq d$. Because $c(C_{y}') \subset X_y$ and $\dim X_y = d$, then $c(C_{y}')$ is an irreducible component of $X_y$ which proves the first claim of the theorem. 

Next, assume $\mathcal{M}$ is locally a direct summand of $\mathcal{F}$ off $S$. Set $X=\mathrm{Specan}(R)$ and $Y=\mathrm{Specan}(Q)$. Then morphism $h$ induces a ring homomorphism $h^{\#}\colon Q \rightarrow R$. Denote by $\mathfrak{n}_y$ the image under $h^{\#}$ of the ideal of $y$ in $Q$. Consider the homomorphism
$$\phi_{y} : \mathcal{R}(\mathcal{M})/\mathfrak{n}_y\mathcal{R}(\mathcal{M}) \rightarrow \mathrm{Sym}(\mathcal{F}(y))$$
Denote its kernel by $I_{\phi_{y}}$. Observe that $$(\mathcal{R}(\mathcal{M})/\mathfrak{n}_y\mathcal{R}(\mathcal{M}))/I_{\phi_{y}}= \mathcal{R}(\mathcal{M}(y)).$$
Let's identify $I_{\phi_{y}}$. Consider the homomorphism $$\tilde{\phi}_{y}: \mathcal{R}(\mathcal{M}) \rightarrow \mathrm{Sym}(\mathcal{F}(y)).$$
We have $\mathrm{Ker}(\tilde{\phi}_{y}) =\mathfrak{n}_y\mathrm{Sym}(\mathcal{F}) \cap \mathcal{R}(\mathcal{M})$. By definition $I_{\phi_{y}}$ is the kernel of $\phi_{y}$. As $\tilde{\phi}_{y}$ factors through $\phi_{y}$ we get
\begin{equation}\label{vertical eq}
I_{\phi_{y}} = (\mathfrak{n}_{y} \mathrm{Sym}(\mathcal{F}) \cap \mathcal{R}(\mathcal{M}))/\mathfrak{n}_y \mathcal{R}(\mathcal{M}).
\end{equation}
Because the source of $\phi_{y}$ is supported on $X_y$, then $\mathrm{Supp}(I_{\phi_{y}}) \subset X_y$. Let $x \in X_y$
with $x \not \in S_y$. Then $\mathcal{M}$ is locally a direct summand of $\mathcal{F}$ at $x$. Write $\mathcal{F}_{x} = \mathcal{M}_{x} \oplus L(x)$. The formation of symmetric and Rees algebras commutes with localization, hence $$(\mathfrak{n}_y\mathrm{Sym}(\mathcal{F}))_{x} = \mathfrak{n}_{y}\mathrm{Sym}(\mathcal{F}_{x}) \ \text{and} \ \mathcal{R}(\mathcal{M})_{x}=\mathcal{R}(\mathcal{M})_{x}.$$
On the other hand, $$\mathfrak{n}_y\mathrm{Sym}(\mathcal{F}_{x}) = \mathfrak{n}_{y}\mathrm{Sym}(\mathcal{M}_{x}) \otimes\mathfrak{n}_{y}\mathrm{Sym}(L(x)).$$
Because $\mathcal{M}_{x}$ is free we have $\mathrm{Sym}(\mathcal{M}_{x})=\mathcal{R}(\mathcal{M}_{x})$.
Hence, locally at $x$ the ideals $\mathfrak{n}_{y}\mathcal{R}(\mathcal{M})$ and $\mathfrak{n}_y\mathrm{Sym}(\mathcal{F}) \cap \mathcal{R}(\mathcal{M})$ agree. Finally, we obtain that if $I_{\phi_{y}}$ is nonzero, then it is supported at points from $S_{y}$ only.
In particular, $I_{\phi_{y}}$ vanishes locally at the minimal primes of $\mathcal{R}(\mathcal{M})/\mathfrak{n}_{y}\mathcal{R}(\mathcal{M})$ because as we showed above each of these minimal primes contracts to a minimal prime of $X_y$. Therefore, $C(y)$ and $C_y$ differ by vertical embedded components supported over $S_y$. This proves (\ref{fund. cycl.}).
\end{proof}
\begin{remark}\label{support2} \rm Note that in general without assuming the bound on the dimension of $c^{-1}0$, the proof above shows that $C(y)=(C_y-W)^{-}$ for any $y \in Y$ where $W$ is the union of irreducible components of $C_y$ surjecting on $S_y$.

A more general version of the theorem above without assuming that $S$ is finite over $Y$ can be derived using Bertini's theorem for extreme morphisms from \cite{Rangachev}. The direct summand assumption can be relaxed at the expense of mild hypothesis on $X$ as remarked at the end of Sct.\ $2$ in \cite{Rangachev}. 

\end{remark}

\section{The $A_f$ condition and the main result}
Let $(X,0)$ be a reduced complex analytic set germ with $X=G^{-1}(0)$ where $G\: (\mathbb{C}^{n+k},0) \to (\mathbb{C}^p,0)$. Assume $(Y,0) \subset (X,0)$ is smooth subgerm of dimension $k$. Choose an embedding of $(X,0)$ in $\mathbb{C}^{n+k}=\mathbb{C}^{n} \times \mathbb{C}^{k}$, so that $(Y,0)$ is represented by $0 \times V$, where $V$ is an open neighborhood of $0$ in $\mathbb{C}^{k}$. Let $\mathrm{pr}: \mathbb{C}^{n} \times \mathbb{C}^{k} \rightarrow \mathbb{C}^{k}$ be the projection, $i\:\mathbb{C}^{k}\rightarrow \mathbb{C}^{n} \times \mathbb{C}^{k}$ be the inclusion $i(y)=(0,y)$ and $\pi_Y$ be the retraction $i\circ \mathrm{pr}_{|X}$. View $X$ as the total space of the family $\pi_Y: (X,0) \rightarrow (Y,0)$. For each closed point  $y \in Y$ set $X_y: = X \cap \pi^{-1}_Y(y)$.

The {\it Jacobian module} $JM(X)$ of $X$ is the submodule of ${\mathcal O}^p_{X}$ generated by the partial derivatives of $G$. It is a direct summand of ${\mathcal O}^p_{X}$ locally off the singular locus of $X$. Denote the smooth part of $X$ by $X_{\mathrm{sm}}$. Suppose $f$ is an analytic function on $X$ that is a submersion on $X_{\mathrm{sm}}$. Then the singular locus $S(f)$ of $f$ is contained in $X_{\mathrm{sing}}$. Denote by $\tilde{f}$ an extension of $f$ to the ambient space. Define $H=(G,\tilde{f})$, and let $JM(H)$ denote the ${\mathcal O}_{X,0}$-module defined by the partial derivatives of $H$. Note that $JM(H)$ is independent of the the choice of extension of $f$ by the discussion in the beginning of Sct.\ $5$ in \cite{G-K}. Finally, denote by $c$ the structure morphism $c \: C(X,f) \rightarrow X$ and by $C(Y)$ the conormal space of $Y$ in $\mathbb{C}^{n+k}$. 

We say that $A_f$ condition holds for the pair $X_{\mathrm{sm}},Y$ at $0$ if  $f(Y)=0$ and $Y$ lies in every hyperplane obtained as a limit of tangent hyperplanes to a level hypersurface at a point $x \in X_{\mathrm{sm}}$ as $x$ approaches $0$.

We review briefly the connection between the theory of integral closure of modules and Thom's $A_f$ condition. Recall that given a submodule $\mathcal{M}$ of a free ${\mathcal O}_{X,0}$ module $\mathcal{F}$, we say that $u \in \mathcal{F}$ is {\it strictly dependent} on $\mathcal{M}$ and we write $ u \in \mathcal{\sdep M}$, if for all analytic path germs $\phi:(\mathbb C,0)\to(X,0)$, $\phi^*u$ is contained in $\phi^*(\mathcal{M})m_1$, where $m_1$ is
the maximal ideal of $\mathcal{O}_{\mathbb{C},0}.$

\begin{Prop}\label{A_f} Assume $f(Y)=0$. Then the following are equivalent
\begin{itemize}
\item[i)]The $A_f$ condition holds for the pair $X_{\mathrm{sm}},Y$ at $0$.
\item[ii)] $c^{-1}(Y)\subset C(Y)$.
\item[iii)] $\pd H{y_j}\in\sdep{JM(H)}$ for all $j=1, \ldots, k.$
\end{itemize}
\end{Prop}
\begin{proof}  The equivalence of i) and ii) is obvious; the equivalence of i) and iii) is Lemma 5.1 of \cite{G-K}.
\end{proof}
A similar result holds for the Whitney A condition. The condition we need for our main result is a much weaker version of Whitney A.
\begin {tont} We say that $(X,0) \rightarrow (Y,0)$ satisfies the infinitesimal Whitney A fiber condition at $0$ if $\pd G{y_j}\in\sdep{JM(X_0)}$ for all $j=1, \ldots, k.$\end{tont}

This condition is equivalent to asking that limiting tangent hyperplanes to $X$ along curves on  $X_0$ contain the tangent space to $Y$ (cf.\ Lemma 4.1 in \cite{G-K}). So it is much weaker than asking that Whitney A hold for the pair $(X_{\mathrm{sm}}, Y)$ at $0$, which would require looking at all curves on $X$ passing through the origin.

We show how weak the infinitesimal Whitney A condition is by considering a family of examples due to Trotman (see Prop.\ 5.1, p.\ 147 in \cite{Trot}). In these examples, the members of the families are the same, but the total space is different. The examples were used to show that a necessary and sufficient fiberwise condition for Whitney A was impossible.

\begin{exa} Consider the family of plane curves with parameter $y$ given by $w^a-y^bv^c-v^d=0$, so $X_0$ is the curve defined by $w^a-v^d=0$ and $y=0$. Then the infinitesimal Whitney A fiber condition holds at $0$ if $b>1$, for all $a,c,d$, because on $X_0$ we have $\pd G{y}=0$. If $b=1$, the condition  holds if $c>\mathrm{min}\{d-d/a, d-1\}$. 

Indeed, let $\phi:(\mathbb C,0)\to(X_0,0)$ be a curve, and let $t$ be the generator for the maximal ideal of $\mathcal{O}_{\mathbb{C},0}$. Write $\phi^{*}(w)=t^{\alpha_1}w_{1}(t)$ and $\phi^{*}(v)=t^{\beta_1}v_{1}(t)$, where $w_{1}(t)$ and $v_{1}(t)$ are units in $\mathcal{O}_{\mathbb{C},0}$. Because $X_0$ is cut out by $w^a-v^d=0$, then $a\alpha_1=d\beta_1$. The infinitesimal Whitney A fiber condition holds if $\phi^{*}(v^c) \in t(\phi^{*}(w^{a-1}), \phi^{*}(v^{d-1}))$, or equivalently if $c\beta_1 > \mathrm{min} \{(a-1)\alpha_1,(d-1)\beta_1 \}$, which is the same as  $c>\mathrm{min}\{d-d/a, d-1\}$ because $\alpha_1 = \frac{d\beta_{1}}{a}$. 
\end{exa}
   
Preserve the setup from the beginning of the section. The following theorem is the main result of our paper.    
\begin{theorem} \label{main} Suppose $X$ and $X_y$ are equidimensional. Assume the singular locus of $X$ is $Y$. Suppose $f$ is a function on $X$ such that $f$ is a submersion on  $X-Y$ and $f(Y)=0$. Suppose  $\dim C(X,f)_0<n$, and the infinitesimal Whitney A fiber condition holds at $0$. Then $A_f$ holds for the pair $(X-Y,Y)$ at $0$.
\end{theorem}

\begin{proof} We need to show that $C(X,f)_0 \subset C(Y)$. 
By Theorem \ref{conormal fiber} we know the components of $ C(X,f)_0$ are of two types: the blow-up components $\pi(E_{\tilde{f}})_{0}$ and the join components $d\tilde{f}(0)*C(X)_0$ if $d\tilde{f}(0)\ne 0$. We will show that the irreducible components of  $\pi(E_{\tilde{f}})_0$ are contained in $C(Y)$, while the join components are controlled by the infinitesimal Whitney A fiber condition. 

We claim that $\pi(E_{\tilde{f}})$ is of pure dimension equal to  $\dim \mathrm{Bl}_{\mathrm{im} \ d\tilde{f}}T_{X}^{*}U -1 = n+k-1$, where $U$ is a neighborhood of $0$ in $\mathbb{C}^{n+k}$ that contains $X$. Indeed, by Lemma \ref{exceptional divisor} \rm{(i)} $\pi(E_{\tilde{f}})$ 
is isomorphic to $E_{\tilde{f}}$, and $E_{\tilde{f}}$ is of pure dimension $n+k-1$. 

Because the singular locus of $X$ is $Y$ and because $f$ is a submersion on $X-Y$, then $S(f) \subset Y$. Then by Lemma \ref{exceptional divisor} \rm{(ii)} $\pi(E_{\tilde{f}})$ is supported over $Y$. By assumption,  $\dim \pi(E_{\tilde{f}})_0 \leq n-1$. Hence, by upper semi-continuity $\dim \pi(E_{\tilde{f}})_y \leq n-1$ for each $y$ in a neighborhood of $0$. But $\dim Y = k$. Thus the dimension formula implies that the irreducible components of $\pi(E_{\tilde{f}})$ surject onto $Y$. Since the $A_f$ condition holds generically on $Y$, each irreducible component of $\pi(E_{\tilde{f}})$ is generically contained in $C(Y)$. Hence each such component lies in $C(Y)$. In particular, 
$\pi(E_{\tilde{f}})_0$ is contained in $C(Y)$. 

Now we turn to the join components. Consider a component $Z$ of the fiber of $C(X)_{0}$. Since  $\dim d\tilde{f}(0)*Z<n$, then $\dim Z<n-1$. Since $Y\subset f^{-1}(0)$, then $d\tilde{f}(0)\in C(Y)$. So it suffices to show that $Z\subset C(Y)$. Apply Theorem \ref{components} with 
$\mathcal{M}:=JM(X)$. Then $C$ is the conormal space $C(X)$. Also, $C_0$ is $C(X)|X_0$ and
$C(0)$ is the closed subscheme of $C(X)$ that consists of the closure of the pairs $(x,H)$ where $x$ is a smooth point of $X_0$ and $H$ is a tangent hyperplane to $X$ at $x$, and $r=n-1$. Since the dimension of $Z$ is less than $n-1$ then by Thm.\ \ref{components} it follows that $Z$ consists of limits of tangent hyperplanes at $0$ along curves on $X_0$. Thus the infinitesimal Whitney A fiber condition implies that $Z$ is in $C(Y)$.
\end{proof}

The usefulness of the last theorem rests on our ability to control the dimension of $\dim C(X,f)_0$ by numerical means. We give an example improving Thm.\ 5.8 from \cite{G-R} that shows how this works. For another example see Thm.\ 1.8.2 in  \cite{R-thesis}. Recall that $X\subset \mathbb{C}^{n+k}$ is a determinantal singularity if the ideal of $X$ is generated by the minors of fixed size of an $(l+q)\times l$ matrix with entries in $\cO_{n+k}$, and $X$ has the expected codimension. The matrix is called the presentation matrix of $X$ and is denoted  $M_X$.

Denote by $JM(G_y;f_y)$ the restriction to the fiber $X_y$ of the augmented Jacobian module $H$ as defined in the beginning of the section. Denote by $N_D(y)$ the module of first order infinitesimal deformations of $X_y$ coming from the deformations
of the presentation matrix $M_{X_y}$. Let $\Sigma^l$ be the $(l+q) \times l$ matrices of kernel rank $l$. Finally, define $e_{\Gamma}(JM(G_y;f_y),\cO_{n+k}\oplus N_D(y))$ to be the sum of the multiplicity of the pair of modules $(JM(G_y; f_y), \cO_{l+q}\oplus M^*_{ X_y}(JM(\Sigma^l)))$ and the intersection number of the image of $M_{X_y}$ with a polar of $\Sigma^l$ of complementary dimension to $n$.

\begin{theorem}\label{determinantal}   Suppose $(X,0) \rightarrow (Y,0)$ is a family of determinantal singularities with presentation matrix $M_{X}:\mathbb{C}^{n+k}\to \mathrm{Hom} (\mathbb{C}^l,\mathbb{C}^{l+q})$, defined by the maximal minors of $M_{X}$.
Suppose $X = G^{-1}(0)$,  where $G:({\mathbb{C}}^{n+k},0)\to({\mathbb{C}}^p,0)$ with $Y$ a smooth subset of $X$, coordinates  chosen so
that $0 \times \mathbb{C}^k  = Y$. Assume  $X$ is equidimensional with equidimensional fibers of the expected dimension and $X$ is reduced.

Suppose $f\:(X,0)\to (\mathbb{C},0)$ and set $Z=f^{-1}(0)$. Suppose the infinitesimal Whitney A fiber condition holds at $0$ if $f \not \in m_{Y}^2$.

A) Suppose $X_y$ and $Z_y$ are isolated
singularities, suppose the critical locus of $f$ is $Y$.
Suppose $e_{\Gamma}(JM(G_y;f_y),\cO_{n+k}\oplus N_D(y))$ is independent of $y$. 
 Then the union of the singular points
of $f_y$
 is $Y$, and the pair of strata
$(X-Y,Y)$ satisfies Thom's $A_f$ condition.

B) Suppose the critical locus of $f$
 is $Y$ or is empty, and  the pair $(X-Y,Y)$ satisfies Thom's $A_f$
condition.  Then   $e_{\Gamma}(JM(G_y;f_y),\cO_{n+k}\oplus N_D(y))$ is independent of $y$.\end{theorem}
\begin{proof}
The proof follows the lines of the proof of Thm.\ 5.8 in \cite{G-R}. It uses Thm.\ \ref{main} to cover the case when $f \not \in m_{Y}^2$. 

In \cite {G-R}, it is shown that constancy of  $e_{\Gamma}(JM(G_y;f_y),\cO_{n+k}\oplus N_D(y))$ implies that $H$ has no polar variety of dimension $k$. In turn this implies $C(X,f)_0$ has no component of dimension $n$ or more. If $f\in m^2_Y$, this implies the result of \cite {G-R}. If $f \not \in m^2_Y$, then it allows us to use Thm.\ \ref{main} to prove the above strengthening of the result of \cite {G-R}.
\end{proof}

In a similar way, Theorems 5.3, 5.4 of \cite {G1}, and Theorem 5.6 of \cite {G2} can be strengthened, dropping the hypothesis of $f\in m^2_Y$.

\end{document}